\documentclass{amsart}
\usepackage{graphicx} 

\usepackage{stackengine}
    \newcommand\xrowht[2][0]{\addstackgap[0.5\dimexpr#2\relax]{\vphantom{#1}}}
\usepackage{amsmath,amssymb}
\usepackage{array} 
\usepackage{ amssymb }
\usepackage{amsfonts}
\usepackage{amsthm}
\usepackage{mathabx}
\usepackage{soul}
\usepackage{graphics}
\usepackage{graphicx}
\usepackage{float}
\usepackage{enumerate}

\usepackage{marginnote}

\usepackage{ mathrsfs }

\usepackage[
colorlinks=true,
linkcolor=blue,
 citecolor=blue,
  urlcolor=blue,
     pagebackref,
]{hyperref}

\usepackage[normalem]{ulem}
\usepackage{bm}
  \newcolumntype{P}[1]{>{\centering\arraybackslash}p{#1}}
    \newcolumntype{M}[1]{>{\centering\arraybackslash}m{#1}}

\title{The exact dimension of Liouville numbers: The Fourier side}

\author{Iván Polasek}
\address[Iván Polasek]{Department of Mathematics,
University of Buenos Aires and IMAS-CONICET-UBA}
\email{ipolasek@dm.uba.ar}
\author{Ezequiel Rela}
\address[Ezequiel Rela]{Department of Mathematics,
University of Buenos Aires and IMAS-CONICET-UBA}
\email{erela@dm.uba.ar}
\address[Ezequiel Rela]{Guangdong Technion Israel Institute of Technology, Shantou, China. } 
\email{ezequiel.rela@gtiit.edu.cn}


\newtheorem{theorem}{Theorem}[section]
\newtheorem{lemma}[theorem]{Lemma}
\newtheorem{proposition}[theorem]{Proposition}
\newtheorem{corollary}[theorem]{Corollary}

\theoremstyle{definition}
\newtheorem{definition}[theorem]{Definition}

\newtheorem*{remark}{Remark}

\usepackage[usenames,dvipsnames]{xcolor}

\newenvironment{proof 1}{\paragraph{Proof:}}{\hfill$\square$}

\graphicspath{{Figuras/}}

\begin{document}
\begin{abstract}
    In this article we study the generalized Fourier dimension of the set of Liouville numbers $\mathbb{L}$. Being a set of zero Hausdorff dimension, the analysis has to be done at the level of functions with a  slow decay at infinity acting as control for the Fourier transform of (Rajchman) measures supported on $\mathbb{L}$. We give an almost complete characterization of admissible decays for this set in terms of comparison to power-like functions. This work can be seen as the ``Fourier side'' of the analysis made by Olsen and Renfro regarding the generalized Hausdorff dimension using gauge functions. We also provide an approach to deal with the problem of classifying oscillating candidates for a Fourier decay for $\mathbb{L}$ relying on its translation invariance property.
\end{abstract}
\maketitle

\section{Introduction and Main Results}\label{sec:intro}

The study of the Fourier decay properties of measures supported on sets of fractal dimension is a central problem in analysis exploiting the interplay between  harmonic analysis and geometric measure theory. The decay rate of $\hat\mu$ for a fractal measure has a fundamental role in restriction theorems for the Fourier transform, after the groundbreaking work of \cite{moc00}. A classical problem related to the Fourier decay is the famous uniqueness problem for trigonometric series (interesting and new results for the case of self-similar sets can be found in \cite{Li-Sahlsten}, \cite{VY}; a mandatory reference on the uniqueness problem is the work of Kahane and Salem \cite{ks63}). 

Additionally, it is relevant for the study of the existence of normal numbers on fractal sets as can be seen in the recent works \cite{FW-2024},\cite{FHR-2024}.

The optimal power-like decay is used to define the Fourier dimension of a set $E \subset \mathbb{R}^n$ as
\begin{equation}\label{eq:dimF}
    \dim_F(E)=\sup\left\{\beta\in[0,n]:\exists \mu\in\mathcal{P}(E): |\hat\mu(\xi)|\lesssim(1+|\xi|)^\frac{-\beta}{2}\right\}.
\end{equation}

It is a deep and important problem to detect, for a given set $E\subset \mathbb{R}^n$, the optimal decay rate for a probability measure $\mu \in\mathcal{P}(E)$. There is a well known constraint on the range of decay rates given by the Hausdorff dimension ($\dim$) of the support of $\mu$, namely
\begin{equation}\label{eq:dimF le DimH}
    \dim_F(E)\le \dim(E).
\end{equation}
In the case when equality holds for a set $E$, it is called a Salem set (see 
 for example \cite{blu98}, \cite{FraHam}  for more details on this).
This means that in particular for sets of zero Hausdorff dimension there cannot be a power-like decay for any measure supported on it. There are, however, examples of zero dimensional sets supporting Rajchman measures, i.e., measures $\mu$ such that $\hat\mu\to 0$ at infinity. 

A relevant case to our purpose on this article is the case of Liouville numbers defined as follows: 
\begin{equation}\label{def:Liouville}
    \mathbb{L} = \{ x \in \mathbb{R}\setminus \mathbb{Q}: \forall n \in \mathbb{N},  \exists \ q \in \mathbb{N} \ \text{such that } \| qx \| < q^{-n} \},
\end{equation}
where $ \| x \| = \min_{m \in \mathbb{Z}} |x-m|$.
In \cite{blu00} a specific construction of a measure supported on $\mathbb{L}$ vanishing at infinity is shown. 

The main purpose of the present article is to address the problem of detecting if a given (slowly enough) decay rate is a plausible candidate as a control for the Fourier transform of a measure supported on $\mathbb{L}$. In that direction, \eqref{eq:dimF le DimH} looks inefficient, since $\dim(\mathbb{L})=0$ and therefore nothing can be obtained from there. For the particular case of zero dimensional sets (and actually also for any dimension in general) it is useful to consider the finer notion of size defined in terms of dimension functions $h:[0,+\infty)\to [0,+\infty)$ and their corresponding  Hausdorff measures $\mathcal{H}^h$:
\begin{equation}\label{def:Hausdorff-h}
   \mathcal{H}^h(E):= \lim_{\delta \rightarrow 0} \inf \left\{ \sum h(\textnormal{diam}(E_i)): E_i \ \textnormal{open}, \bigcup E_i \supset E, \textnormal{diam}(E_i) \le \delta \right\}.
\end{equation}

We can use this to refine the notion of size and try to detect the exact dimension function that is accurate to measure a given set. In particular, we say that $E$ is an $h$-set if $0<\mathcal{H}^h(E)<\infty$. It is not always possible to find a dimension function for which this holds. In this case, a common situation is that a ``dimensional partition'' is provided. That is, for a given set $E$ and a reasonable family of dimension functions, it is possible to provide a cut-point which allows to discriminate between those functions such that $\mathcal{H}^h(E)=0$ from those such that $\mathcal{H}^h(E)=+\infty$. This is precisely the case proved by Olsen and Renfro in \cite{ols05}\cite{or06} for the Liouville numbers, where they proved the following result.

\begin{theorem}[Olsen \& Renfro]\label{thm:Olsen-Renfro}
   Let $h$ be an arbitrary dimension function. Define the function $\Gamma_h$ by
   \begin{equation*}
       \Gamma_h (r) = \inf_{0<s\le r} r \frac{h(s)}{s}.
   \end{equation*}
   Then
   \begin{enumerate}
       \item If $\displaystyle \limsup_{r \searrow 0} \frac{\Gamma_h(r)}{r^t} = 0$ for some $t>0$, then $\mathcal{H}^{h}(\mathbb{L})=0$.
       \item If $\displaystyle \limsup_{r \searrow 0} \frac{\Gamma_h (r)}{r^t} >0$ for every $t>0$, then $\mathcal{H}^h$ is not $\sigma$-finite on $\mathbb{L}$.
   \end{enumerate}
\end{theorem}

Elekes and M\'ath\'e \cite{ek06} went further on this and they proved that this 0/$\infty$ law is not special to Hausdorff measures, but it is an incommensurability property of Liouville numbers: no translation invariant Borel measure will assign a positive and $\sigma$-finite value to the set of Lioville numbers.

In the same spirit as in Theorem \ref{thm:Olsen-Renfro}, we address the dual problem of finding the exact dimension of Liouville numbers on the Fourier side. Inspired by the work of Olsen and Renfro and also by the construction of Bluhm in \cite{blu00}, we provide a description of the allowed Fourier decay for a measure supported on the Liouville set. 

To state our results, we introduce the following Fourier set associated to a given measurable set $E\subset \mathbb{R}$. We will adopt the standard terminology and notation and say that $f$ is ``smaller'' than $g$  if $f = O(g)$ at infinity, which means that  there exist $M>0, c>0$ such that $|f(|x|)| \le c |g(|x|)|$ when $|x| \ge M$.

\begin{definition}\label{def:FourierSet-EquivClasses}
    Let $E \subset \mathbb{R}$. Define 
    \begin{equation}
    \tilde{\mathcal{F}}(E) = \left\{ f: \mathbb{R}_{\ge0}\to \mathbb{R}_{\ge0} / \ \exists\ \mu \in  
\mathcal{P}(E) \ \textnormal{and} \ \hat{\mu}= O(f) \ \right\}.
\end{equation}
\end{definition}

\begin{definition}\label{def:fourier-set}
    Let $E \subset \mathbb{R}$. Define its Fourier set as
    \begin{equation*}
    \mathcal{F}(E) = \frac{\tilde{\mathcal{F}}(E)}{\sim},
\end{equation*}
where the equivalence relation $\sim$ between elements of $\tilde{\mathcal{F}}(E)$ is defined as
\begin{align*}
    f \sim g \iff  f= O(g) \ \textnormal{and} \ g=O(f).
\end{align*}
In this case we will say that $f$ and $g$ are \textit{asymptotically close}.
\end{definition}

The benefit of working with $\mathcal{F}(E)$ is that its elements, the equivalence classes $[f]$ of functions $f$ under the relation $\sim$, codify information only about the decay of $f$ in $+\infty$ without tying themselves to any specific values of the function. This is useful because it is only the decay of a function $f$ which determines if it belongs to $\tilde{\mathcal{F}}(E)$ (this is obvious from the definition). For example, we might say that $\left[\frac{1}{\log \log (\xi)}\right] \in \mathcal{F}(E)$ even if the function is not defined for every $\xi \in \mathbb{R}^+_0$, by thinking of $\left[\frac{1}{\log \log (\xi)}\right]$ as the class $[f]$ with representative
\begin{equation*}
    f(\xi)= \left\{ \begin{array}{ccc}
    1 & \textnormal{if} & 0 \le \xi \le e^e \\
    \\
    \frac{1}{\log \log (\xi)} & \textnormal{if} & \xi>e^e.
    \end{array}
    \right.
\end{equation*}

Defining $\mathcal{F}(E)$ in this way will be useful in the long run, because it avoids carrying unnecessary information. It is, however, a bit tedious, so we will sometimes identify an element $[f] \in \mathcal{F}(E)$ with some convenient representative $f$. It is important to notice that, since $\sim$ relates functions which behave similarly enough in infinity, proving that a class $[f]$ belongs to $\mathcal{F}(E)$ is achieved simply by proving that one representative $f$ belongs to $\tilde{\mathcal{F}}(E)$.

\subsection{The Fourier set of \texorpdfstring{$\mathbb{L}$}{L}}\label{sec:fourier-set}
We now present our first main result, concerning the Fourier set of Liouville numbers $\mathbb{L}$, namely $\mathcal{F}(\mathbb{L})$. Bluhm proved in \cite{blu00} that $\mathbb{L}$ supports a Rajchman measure, which in our terminology is saying that there exists some function in $\mathcal{F}(\mathbb{L})$ that tends to $0$ at infinity. By the elementary dimension estimate \eqref{eq:dimF le DimH} it is clear that only slowly decaying functions  may belong to $\mathcal{F}(\mathbb{L})$. By this we mean slower than any power-like decay; perhaps the first and most intuitive example of such functions are $\frac{1}{\log^{p}\xi}$ for $p>0$.
The purpose of the following theorem is to provide a Fourier dimensional partition of the set of decreasing or continuous functions according to their membership in $\mathcal{F}(\mathbb{L})$. In particular, and roughly speaking, one would like to determine what is the fastest possible Fourier decay allowed for the Liouville numbers (among all slowly decaying functions, of course). The first main result in this article is the following.

\begin{theorem}\label{thm:fourier-set-liouville-decreasing}
    Let $f: \mathbb{R}_{\ge 0}\rightarrow {\mathbb{R}_{\ge 0}}$ be a bounded function that is either decreasing or continuous. Then 
    \begin{enumerate}[(i)]
        \item If 
           \begin{equation}\label{eq:decay-condition}
    \limsup_{\xi \rightarrow +\infty} \frac{\xi^{-\alpha}}{f(\xi)}=0 \ \forall \ \alpha>0,
\end{equation}
then $f \in \mathcal{F}(\mathbb{L})$.
\item If 
\begin{equation}\label{eq:decay-condition-2}
    \liminf_{\xi \rightarrow +\infty} \frac{\xi^{-\alpha_0}}{f(\xi)}>0 \ \textnormal{for some} \ \alpha_0>0,
\end{equation}
then $f \notin \mathcal{F}(\mathbb{L})$.
    \end{enumerate}

\end{theorem}

\begin{remark}
We note here that since the universe of functions that are allowed in Theorem \ref{thm:fourier-set-liouville-decreasing} are either continuous or decreasing, we may say something more about the possible zeroes of those functions. One trivial observation is that for decreasing functions, no zeroes are allowed, since in that case the function will become identically zero from some point onwards and therefore obviously does not belong in $\mathcal{F}(\mathbb{L})$. On the other hand, for continuous functions that are not decreasing, we note that the set of zeroes has to be bounded for condition \eqref{eq:decay-condition} to hold. 
\end{remark}

This result provides a large class of admissible Fourier decays for $\mathbb{L}$. Recall that Bluhm constructed a measure $\mu_\infty$ with a good decay given by $|\widehat{\mu}_\infty (\xi)| \lesssim B(\xi)$, where $B$ is a control function given by a series expansion decaying to zero at infinity (see \eqref{eq:Bluhm-function}).
An easy computation shows, in particular, that the estimate
\begin{equation}\label{eq:Bluhm-estimate}
|\widehat{\mu}_\infty (\xi)| \lesssim \frac{1}{\log^p (|\xi|)}    
\end{equation}
holds for any $p>0$, providing a first example of a family of explicit functions in $\mathcal{F}(\mathbb{L})$.  As an interesting consequence, this implies that $\mathbb{L}$ contains numbers which are normal to every base (see \cite{PVZZ}). It remains, however, the question about the optimality of Bluhm's result.  A deeper question is about the optimal possible decay for $\mathbb{L}$ in general.
In that direction, the next corollary shows that 
our result allows to push any function satisfying \eqref{eq:decay-condition} towards the threshold for admissible decays by taking any power of it.
It follows from an easy but also important observation about the algebraic structure given by condition \eqref{eq:decay-condition}. 

\begin{remark}\label{rem:algebra-structure}
If a function $f$ satisfies condition \eqref{eq:decay-condition} then  $f^n$ also   satisfies condition \eqref{eq:decay-condition} for every $n \in \mathbb{N} $. 
\end{remark}

As a consequence, we obtain that Bluhm's result can be immediately improved.

\begin{corollary}\label{rem:Bluhm powers}
Let  $B$ be the function (vanishing at infinity) obtained by Bluhm in \cite{blu00} such that there is a Rajchman measure $\mu_\infty$ supported in $\mathbb{L}$ with $|\widehat{\mu_\infty}(\xi)| \lesssim B(\xi)$ (again, we refer to \eqref{eq:Bluhm-function}). Then $B^n\in \mathcal{F}(\mathbb{L})$ for every $n \in \mathbb{N} $.
\end{corollary}

We now introduce two new auxiliary functions ($\tau$ and $\phi$ in the next theorem) that will prove to be important for understanding the behaviour of a \textbf{positive} decay $f$ . We will later prove a series of results that explore this in more detail.

\begin{theorem}\label{thm:equivalent-decay-conditions}
    Let $f: {\mathbb{R}_{\ge 0}} \longrightarrow {\mathbb{R}_{> 0}}$ be as be as in Theorem \ref{thm:fourier-set-liouville-decreasing}. Define the functions $\tau(\xi)$ and $\phi(\gamma)$ as
    \begin{equation*}
        \tau(\xi) = -\frac{\log f(\xi)}{\log \xi} \quad \textnormal{and} \quad \phi(\gamma) = -\log(f(e^\gamma)).
    \end{equation*}
    Then the following conditions:
    \begin{equation}\label{eq:tau-guarantees-inside}
        \limsup_{\xi \rightarrow + \infty} \tau(\xi) = 0
    \end{equation}
    and
    \begin{equation}\label{eq:phi-guarantees-inside}
        \limsup_{\gamma \rightarrow + \infty} \frac{\phi(\gamma)}{\gamma} = 0
    \end{equation}
    are both equivalent to \eqref{eq:decay-condition}, and as such they both imply $f \in \mathcal{F}(\mathbb{L})$.

    Similarly, the conditions
    \begin{equation}\label{eq:tau-guarantees-outside}
        \liminf_{\xi \rightarrow + \infty} \tau(\xi) > 0
    \end{equation}
    and
    \begin{equation}\label{eq:phi-guarantees-outside}
        \liminf_{\gamma \rightarrow + \infty} \frac{\phi(\gamma)}{\gamma} > 0
    \end{equation}
    are both equivalent to \eqref{eq:decay-condition-2}, and as such they both imply $f \notin \mathcal{F}(\mathbb{L})$.
\end{theorem}

Our next result presents some sort of ideal-type property of the Fourier set of $\mathbb{L}$, under some additional conditions. The point here is to extend our analysis to address another question related to functions which do not fall under any of the two conditions of Theorem \ref{thm:fourier-set-liouville-decreasing}. An example of such function would be 
\begin{equation}\label{eq:example-cos/log}
f(\xi) = \frac{|\cos(2\pi \xi)|}{\log (\xi)}.
\end{equation}
In this example, $f$'s decay is somewhat comparable to that of $\frac{1}{\log (\xi)}$, which belongs in $\mathcal{F}(\mathbb{L})$, but $f$ fluctuates in a way that could allow us to say that it has a faster decay. To be abundantly precise, condition \eqref{eq:decay-condition} fails because of the unbounded set of zeroes of the function $f$.
Additionally, condition \eqref{eq:decay-condition-2} fails because
\begin{equation*}
    \lim_{k \rightarrow \infty} \frac{\left(\frac{k}{2}\right)^{-\alpha}}{f\left(\frac{k}{2}\right)} = \lim_{k \rightarrow \infty} \left( \frac{k}{2} \right)^{-\alpha} \log\left( \frac{k}{2} \right) = 0 \ \forall \ \alpha >0.
\end{equation*}
This responds to the fact that that $f$ fluctuates between behaving like the function $\frac{1}{\log (\xi)}$, which belongs to $\mathcal{F}(\mathbb{L})$, and the constantly zero function, which clearly does not belong to $\mathcal{F}(\mathbb{L})$.

Although this is not enough to decide in full generality if a function with such behaviour belongs in $\mathcal{F}(\mathbb{L})$ or not, we can be sure that there are functions like this in $\mathcal{F}(\mathbb{L})$ and that they can be constructed from other functions in $\mathcal{F}(\mathbb{L})$. 

Another interesting example is given by the function
\begin{equation}\label{eq:example-xi^cos}
    f(\xi) = \xi^{- |\cos(\pi \xi)|}.
\end{equation}
In a similar fashion, this function is sometimes comparable to $\xi^{-1/2}$ and some other times comparable to 1. This means that neither condition \eqref{eq:decay-condition} nor condition \eqref{eq:decay-condition-2} are satisfied.

We will analyze each one of the previous examples and show that $\frac{|\cos(2\pi \xi)|}{\log (\xi)}\in \mathcal{F}(\mathbb{L})$ and that $\xi^{- |\cos(\pi \xi)|}\notin \mathcal{F}(\mathbb{L})$  as a consequence of Theorem \ref{thm:multiplier-fourier-series-period1} below. These two examples show the two possible situations regarding the membership in $\mathcal{F}(\mathbb{L})$ for functions not satisfying neither of the conditions \eqref{eq:decay-condition} or \eqref{eq:decay-condition-2}.
This will turn out to be a consequence of the fact that $\mathbb{L}$ is invariant under integer translations.

\begin{theorem}\label{thm:multiplier-fourier-series-period1}
    Let $f \in \mathcal{F}(\mathbb{L})$. Let $g: \mathbb{R} \rightarrow \mathbb{R}$ be a function which satisfies the following conditions:
    \begin{enumerate}[(1)]
        \item $g$ is even and $1$-periodic.
        \item $g$ is not identically equal to 0.
        \item $g$ has a $1$-periodic Fourier series $S_N(\xi)$ with nonnegative coefficients.
        \item $S_N$ converges to $g$ pointwise.
    \end{enumerate}
    Then $|g|f \in \mathcal{F}(\mathbb{L})$.
\end{theorem}
By using this result, we will provide an answer to the question above about the two examples of fluctuating functions in Proposition \ref{prop:simple-cases-in-out}.

\subsection{Outline}\label{sec:outline}
This article is organized as follows. In Section \ref{sec:proofs} we present the proof of the main result Theorem \ref{thm:fourier-set-liouville-decreasing} and also the proof of several technical reformulations of it. In Section \ref{sec:SI-PM} we move forward to the analysis of oscillating decays and provide a result exploiting the integer translation invariance of $\mathbb{L}$.
\ 

\ 

\

\section{Proofs}\label{sec:proofs}
For the sake of completeness, it will be useful to review Bluhm's construction of a Rajchman measure supported on $\mathbb{L}$; that is,  a measure $\mu_{\infty}$ such that 
\begin{equation*}
    \widehat{\mu_{\infty} }(\xi) \xrightarrow{|\xi|\rightarrow +\infty} 0.
\end{equation*}

Let $\mathbb{P}_M $ denote the set of all prime numbers between $M$ and $2M$, with $M$ being a positive integer. Let $\| x \| = \min_{m \in \mathbb{Z} } |x-m| $.  For any given sequence $(M_k)_k$ of increasing natural numbers such that $M_1 < 2M_1 < M_2 < 2M_2 < M_3 < \ldots $ we define the set 
\begin{equation*}
    S_{\infty} = S_{\infty} \left( (M_k)_k \right) = \bigcap_{k=1}^{\infty} \bigcup_{p \in \mathbb{P}_{M_k}} \{ x \in [
0,1]: \| px \| \le p^{-1-k}  \}.
\end{equation*}

The author constructs a measure $\mu_{\infty}$ in a way to guarantee that it is supported on $S_{\infty} $, and with certain hypotheses on the sequence $(M_k)_k$ it can be proved that $\widehat{\mu_\infty} \xrightarrow{|\xi| \rightarrow +\infty} 0$. This will work for our purposes since, as it can easily be checked, $S_{\infty} \setminus \mathbb{Q} \subset \mathbb{L} $. 

The construction involves defining functions $g_M$. We will not give the full detail, but suffice to say that, for every $k >0$, $M \in \mathbb{N} $, we can define functions $g_M = g_{M,k} $ such that 
\begin{itemize}
    \item $g_M \in C^2(\mathbb{R} ) $,
    \item $g_M$ is $1$-periodic,
    \item supp$(g_M) \subset \bigcup_{p \in \mathbb{P}_M } \{ x: \| px \| \le p^{-1-k} \} $,
    \item $\widehat{g_M}(0)=1$.
\end{itemize}

With this, we are in a position to quote a lemma, proved in \cite{blu98}, about a very important property of the functions $g_M$ that will be of vital importance.

\begin{lemma}\label{lem:Bluhm-generator}
    For every $\psi \in C_0^2 (\mathbb{R})$, $k \in \mathbb{N}$ and $\delta>0$ there exists $M_0 = M_0(\psi,k, \delta)$ such that
    \begin{equation}\label{eq:Bluhm-generator-inequality}
        |\widehat{\psi g_{M,k}}(\xi) - \widehat{\psi}(\xi)| \le \delta \theta_k(\xi) \ \forall \ \xi \in \mathbb{R}
    \end{equation}
    for all $M\ge M_0$, where $\theta_k(\xi) = (1+|\xi|)^{-1/(2+k)} \log(e+|\xi|) \log(e+\log(e+|\xi|))$. 
\end{lemma}

This will allow us to proceed inductively. It will be useful to give a name to the first function of this sequence.

\begin{definition}\label{def:bump}
A  function $\psi_0$ will be called a smooth bump function if
    \begin{itemize}
    \item $\psi_0 \in C_0^2 (\mathbb{R} )$
    \item $\displaystyle\int \psi_0(x)dx = 1 $
    \item $\psi_0 |_{(0,1)} >0 $
    \item $\psi_0 |_{\mathbb{R}\setminus [0,1] } \equiv 0 $,
\end{itemize}
\end{definition}

Starting with a smooth bump function and using Lemma \ref{lem:Bluhm-generator}, the author can choose a specific sequence $(M_k)_k$ of natural numbers and use it to construct a sequence of measures ${\mu_k} $. These measures converge weakly to a measure $\mu_\infty$ which satisfies
\begin{equation}\label{eq:Bluhm-function}
    |\widehat{\mu_\infty}(\xi) | \le \sum_{k=1}^\infty c_k \theta_k(\xi) =: B(\xi),
\end{equation}
for certain specifically defined constants $c_k$. It can be shown that the function $B(\xi)$, which we will  refer to as ``Bluhm's function'', decays to $0$ at infinity, which is what Bluhm wanted to show in his paper. However, Bluhm was not concerned with the rate of decay of the function $B$. We do not get into the full details of the last part of his construction, because it is exactly this part that we will tweak in order to prove our own result,  and as such we will give a detailed explanation there. Note that with our notation $B \in \mathcal{F}(\mathbb{L} ) $. We will exploit this, along with the rate of decay of $B$, to prove Theorem \ref{thm:fourier-set-liouville-decreasing} and show that a large class of functions are in $\mathcal{F}(\mathbb{L} )$. 

Before the proof, let us briefly discuss the reach of the result. The theorem guarantees that under certain hypotheses a decreasing (or continuous) function $f$ belongs to $\mathcal{F}(\mathbb{L})$. Because of the equivalence relation that defines $\mathcal{F}(\mathbb{L})$, we can actually say that this holds for any function which is asymptotically close to some decreasing (or continuous) function. 

We now proceed with the proof of Theorem \ref{thm:fourier-set-liouville-decreasing}.

\begin{proof}[Proof of Theorem \ref{thm:fourier-set-liouville-decreasing}]
We may assume without loss of generality that $f(0) = \frac{1}{2}$ by replacing it, if necessary, with a representative that satisfies the same hypothesis (continuity or monotonicity) that $f$ does.

To prove $(i)$, we will use Lemma \ref{lem:Bluhm-generator} in a slightly different way than Bluhm did. This will allow us to produce Rajchman measures whose decay we can compute explicitly. 
Let $f$ be a function satisfying \eqref{eq:decay-condition}. We already mentioned  that the existence of a sequence $(\xi_k)_k$ such that $\xi_k \nearrow \infty$ and $f(\xi_k) = 0$ denies \eqref{eq:decay-condition}. As such, the set of zeroes of $f$ is bounded, and we can therefore assume that $f$ is positive by conveniently replacing it with a representative that satisfies the same hypothesis that $f$ does. Once this is noted, we are able to state that the quantity
\begin{equation*}
        \rho_k = \left[ \max_{\xi \in \mathbb{R}} \frac{\theta_k(\xi)}{f(|\xi|)} \right] ^{-1}
\end{equation*}
is well defined for every $k \in \mathbb{N}$, where $\theta_k$ is as in Lemma \ref{lem:Bluhm-generator}. Indeed,
\begin{align*}
    \frac{\theta_k(\xi)}{f(|\xi|)} & = \frac{(1+|\xi|)^{-1/(2+k)} \log(e+|\xi|) \log(e+\log(e+|\xi|))}{f(|\xi|)}\\ 
    & = \underbrace{(1+|\xi|)^{-1/[2(2+k)]} \log(e+|\xi|) \log(e+\log(e+|\xi|))}_{A_1 (\xi)} \underbrace{\frac{(1+|\xi|)^{-1/[2(2+k)]}}{f(|\xi|)}}_{A_2 (\xi)}.
\end{align*}

We note here that $A_1 (\xi)\xrightarrow{|\xi|\rightarrow +\infty} 0$ trivially holds. Moreover, we also have that  $A_2 (\xi)\xrightarrow{|\xi|\rightarrow +\infty} 0$ as a consequence of the assumption  \eqref{eq:decay-condition}.  Therefore, there exists $N>0$ such that
\begin{equation*}
    0<\frac{\theta_k(\xi)}{f(|\xi|)} \le 1 \ \textnormal{if} \ |\xi| >N.
\end{equation*}

Now, if $f$ is decreasing it turns out that in $[-N,N]$
\begin{equation*}
    \frac{\theta_k(\xi)}{f(|\xi|)} \le \frac{\theta_k(\xi)}{f(N)},
\end{equation*}
and this function has an absolute maximum in $[-N,N]$ because it is continuous. If $f$ were not decreasing but continuous and positive, a similar bound would hold, although the maximum would occur at a point $\xi_0$ in $[-N,N]$ which might not be precisely $N$. In any case, $\frac{\theta_k(\xi)}{f(|\xi|)}$ has an absolute maximum in $\mathbb{R}$ and $\rho_k$ is well-defined.

We are now in a position to invoke Lemma \ref{lem:Bluhm-generator} and use it to create the desired measure. Let us explain this in detail.

Start with a smooth bump function $\psi_0$, and choose $k=1$ and $\delta = 2^{-1} \rho_1$. The lemma then guarantees the existence of a first $M_1 = M_1 (\psi_0, 1, 2^{-1} \rho_1)$ such that 
\begin{equation*}
    |\widehat{\psi_0 g_{M_1}}(\xi) - \widehat{\psi_0}(\xi)| \le 2^{-1} \rho_1 \theta_1 (\xi).
\end{equation*}

We can proceed inductively and choose $M_k = M_k(\psi_0 G_{k-1},k, 2^{-k} \rho_k )$, such that
\begin{equation}\label{eq:adapted-Bluhm-generator}
    |\widehat{\psi_0 G_{k}}(\xi) - \widehat{\psi_0 G_{k-1} }(\xi)| \le 2^{-k} \rho_k \theta_k (\xi),
\end{equation}
where $G_k = \prod_{i=1}^k g_{M_i}$.

Now construct the set $S_{\infty}$ associated to $(M_k)_k$. Let $\lambda$ be the Lebesgue measure on $\mathbb{R} $ and define the measures
\begin{equation*}
    \mu_k = \psi_0 G_k \lambda, \quad k \in \mathbb{N}_0.
\end{equation*}
Since $\psi_0, g_M$ are all integrable and bounded, it is clear that $(\mu_k)_k$ is a sequence of bounded measures. We will see that $(\mu_k)_k$ converges weakly to some positive finite measure $\mu_{\infty}$ which satisfies all the desired properties. 

First, resume \eqref{eq:adapted-Bluhm-generator} to see that
\begin{align*}
|\widehat{\psi_0 G_{k}}(\xi) - \widehat{\psi_0 G_{k-1} }(\xi)| &\le 2^{-k} \rho_k \theta_k (\xi) \\
&\le 2^{-k} \frac{f(|\xi|)}{\theta_k(\xi)}\theta_k(\xi) \\
&\le 2^{-k} \| f \|_{\infty}.
\end{align*}
This shows that $(\widehat{\mu_k} )_k$ is a Cauchy sequence with respect to $\| \|_{\infty}$, and as such it converges to a continuous function. A version of Levy's continuity theorem for finite measures (see \cite{schi17}, p. 255) then guarantees the existence of a finite measure $\mu_\infty$ which is the weak limit of $(\mu_k)_k$.

Let us see that this measure has the desired decay. Indeed:
\begin{align*}
    |\hat{\mu}_k(\xi)| &\le |\hat{\mu}_0(\xi)| + \sum_{j=1}^k |\hat{\mu}_j(\xi)- \hat{\mu}_{j-1}(\xi)| \\
    & \lesssim |\xi|^{-2} + \sum_{j=1}^k 2^{-j} \rho_j \theta_{j}(\xi)  \lesssim f(|\xi|),
\end{align*}
using that $\rho_k \theta_{k}(\xi) \le f(|\xi|)$ uniformly for all $k\in\mathbb{N}$ (recall the definition of $\rho_k$). The last inequality holds for $|\xi|>N'$ for sufficiently large $N'$, with neither $N'$ nor the constant depending on $k$. Taking the limit as $k \rightarrow \infty$ guarantees that
\begin{equation*}
    |\hat{\mu}_{\infty}(\xi)| \lesssim f(|\xi|).
\end{equation*}

A similar calculation shows that
\begin{equation*}
    |\hat{\mu}_k(0)-1| =  |\hat{\mu}_k(0)-\hat{\mu}_0(0)| \le f(0) = \frac{1}{2},
\end{equation*}
which serves to prove that $\mu_\infty$ has positive mass. 

Finally, it is clear that
\begin{equation*}
    \textnormal{supp}(\mu_k) = \textnormal{supp} (G_k) = \bigcap_{j=1}^k 
    \textnormal{supp} (g_j) \subset \bigcap_{j=1}^k \bigcup_{p \in \mathbb{P}_{M_j} } \{ x: \| px \| \le p^{-1-j} \},
\end{equation*}
and therefore supp$(\mu_\infty) \subset S_\infty$. The fact that $\mu_\infty$ is a Rajchman measure implies that it does not have any atoms. We refer the reader to \cite{lyo95} for a comprehensive survey on Rajchman measures, where the original proof of this fact due to Neder is cited \cite{Neder}. Therefore $\mu_{\infty}(\mathbb{Q})=0$. As such, there exists a closed subset of $S_\infty \setminus \mathbb{Q} \subset \mathbb{L}$ with positive $\mu_\infty$ measure. This finishes the proof of $(i)$.

The proof of $(ii)$ is rather straightforward. Suppose there exists an $\alpha_0>0$ such that
\begin{equation*}
    \liminf_{\xi \rightarrow +\infty} \frac{\xi^{-\alpha_0}}{f(\xi)}>0.
\end{equation*}
This implies the existence of an $M>0$ such that
\begin{equation*}
    f(|\xi|) \le |\xi|^{-\alpha_0} \ \textnormal{when} \ |\xi|\ge M.
\end{equation*}
If it were the case that $f \in \mathcal{F}(\mathbb{L})$, this would imply that $\xi^{-\alpha_0} \in \mathcal{F}(\mathbb{L})$. This in turn would imply that
\begin{equation*}
    \dim(\mathbb{L}) \ge \dim_F(\mathbb{L}) \ge \min \{ 1, 2\alpha_0 \} >0,
\end{equation*}
which is a contradiction since we already know that $\dim (\mathbb{L})=0$.

\end{proof}

Suppose now that we have a measure $\mu$ whose Fourier decay verifies $|\widehat{\mu}(\xi)| \le f(\xi)$, and we want to produce another measure 
$\nu$ whose Fourier decay verifies $|\widehat{\nu}(\xi) | \le g(\xi)$ for some function $g = o(f)$ (i.e. $\lim_{\xi \rightarrow +\infty} \frac{g(\xi)}{f(\xi)} = 0$). We would like to better understand the similarities and differences between $\mu$ and $\nu$. 

In the proof of Lemma \ref{lem:Bluhm-generator}, it is seen that $M_0$ has a decreasing dependence on $\delta$; that is, the smaller the $\delta$ that we provide, the bigger the $M_0$ that will allow \eqref{eq:Bluhm-generator-inequality} to hold for every $M \ge M_0$. In the proof of Theorem \ref{thm:fourier-set-liouville-decreasing}, we use Lemma \ref{lem:Bluhm-generator} in succesive steps to produce a measure (say, $\mu$) with decay controled by a function (say, $f$). The sequence $(\delta_k)_k$, whose $k$-th term is the $\delta$ used in the $k$-th step, is given by 
\begin{equation*}
    \delta_k = \delta_k (f) = 2^{-k} \rho_k (f) = 2^{-k} \left(\max{\frac{\theta_k}{f}}\right)^{-1}.
\end{equation*}
Because of the way $\rho_k$ is defined, the faster a function $f$ decays to zero at infinity, the smaller every $\delta_k (f)$ will be. What this implies is that, when this method is applied to produce a measure $\nu$ with Fourier decay controled by $g$, the faster decay of $g$ will produce a sequence $(N_k)_k$ which will grow faster than the sequence $(M_k)_k$ that $f$ will produce. In turn, the measure $\nu$ will be supported on a set $S_\infty ((N_k)_k)$ which is different than the set $S_\infty ((M_k)_k)$ which supports $\mu$.

As such, these measures are different in a substantial manner. Note that the set $S_\infty ((N_k)_k)$ is larger than the set $S_\infty ((M_k)_k)$ in the intuitive sense that it allows for a faster Fourier decay, a notion that could potentially be made precise with a proper generalization of the notion of Fourier dimension.

We should now recall that a decreasing or continuous function $f$ may verify neither $\eqref{eq:decay-condition}$ nor \eqref{eq:decay-condition-2}, as the examples showed in the introduction. For such a function, Theorem \ref{thm:fourier-set-liouville-decreasing} would not tell us whether $f \in \mathcal{F}(\mathbb{L} )$. This opens up a ``gap'' of functions whose membership in $\mathcal{F}(\mathbb {L}) $ is not described. Besides the aforementioned examples, we exhibit here another interesting decreasing function. For the sake of simplicity, we present a non-continuous step-like construction, but the reader may easily verify that it can be modified to make it continuous.

\begin{remark}
Let us consider the following function:
    \begin{equation}\label{eq:step-function}
    f(\xi) = \sum_{k=0}^{\infty} \frac{1}{a_k} \chi_{[a_k,a_{k+1})}(\xi), \quad \textnormal{where} \ a_k = e^{e^{e^k}}.
\end{equation}
It is easy to see that 

\begin{equation}\label{eq:left-end-of-steps}
    f(a_k) = \frac{1}{a_k}        
    \end{equation}
but
    \begin{equation}\label{eq:right-end-of-steps}
    \lim_{\xi \rightarrow a_k^-} f(\xi) = \exp \left( -\log^{1/e}(a_k) \right).
\end{equation}
\eqref{eq:left-end-of-steps} implies that $f$ does not verify \eqref{eq:decay-condition}, and \eqref{eq:right-end-of-steps} that $f$ does not verify \eqref{eq:decay-condition-2}. Deciding whether $f$ belongs to $\mathcal{F}(\mathbb{L} )$ or not would require further investigation. 
\end{remark}

\begin{figure}[ht]
\centering
    \includegraphics[width=\textwidth]{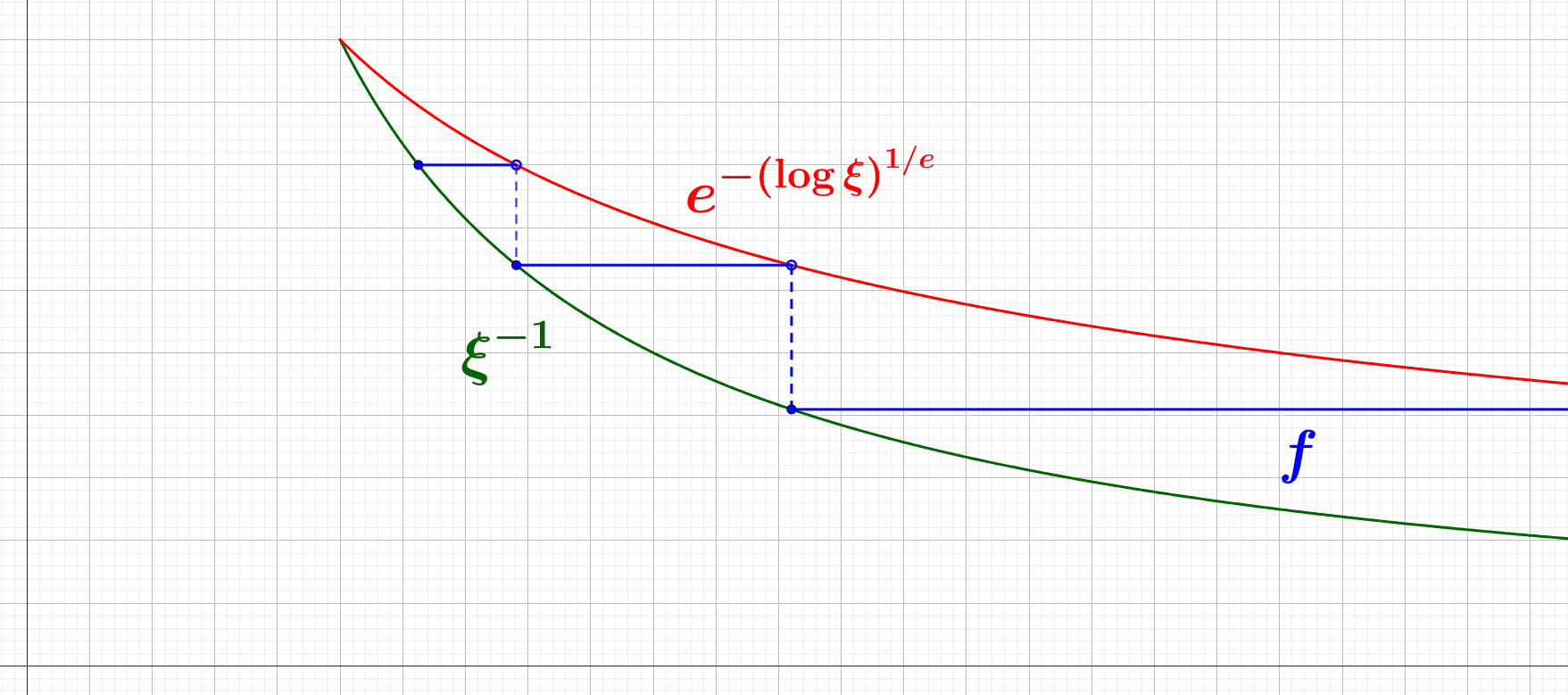}
    \caption{The function $f$ from \eqref{eq:step-function}. Its graph is made of segments which extend horizontally between the graphs of the functions $\xi^{-1} \notin \mathcal{F}(\mathbb{L})$ and $\exp \left( -\log^{1/e}(\xi) \right) \in \mathcal{F}(\mathbb{L})$.}
    \label{fig:remark-function}
\end{figure}

Now we present the proof of Theorem \ref{thm:equivalent-decay-conditions} that reinterprets conditions \eqref{eq:decay-condition} and \eqref{eq:decay-condition-2} in terms of functions related to $f$.

\begin{proof}[Proof of Theorem \ref{thm:equivalent-decay-conditions}]
    We can assume without loss of generality that $\| f \|_\infty \le \frac{1}{2}$. The key idea for proving that \eqref{eq:tau-guarantees-inside} is equivalent to \eqref{eq:decay-condition} is the following:
    \begin{align*}
        \limsup_{\xi \rightarrow + \infty} \frac{\xi^{-\alpha}}{f(\xi)}
        &= \limsup_{\xi \rightarrow + \infty} \frac{e^{- \alpha \log(\xi)}}{e^{\log(f(\xi))}} \\
        &= \limsup_{\xi \rightarrow + \infty} e^{-\alpha \log(\xi) - \log(f(\xi))} \\
        &= \limsup_{\xi \rightarrow + \infty} e^{\log(\xi) (-\alpha + \tau(\xi))}.
    \end{align*}
Due to the fact that $\tau \ge 0$ for sufficiently large $\xi$ and the continuity of the exponential, the limit superior will be zero for every $\alpha>0$ if and only if the limit superior of $\tau$ is zero. The fact that \eqref{eq:tau-guarantees-outside} is equivalent to \eqref{eq:decay-condition-2} follows from an analogous calculation, taking $\alpha_0 = \frac{1}{2} \liminf_{\xi \rightarrow + \infty } \tau(\xi)$. The equivalences regarding $\phi$ follow from applying the change of variables $\xi = e^{\gamma}$ and following the same calculations.

\end{proof}

Each of these interpretations enriches our understanding of the problem in a different way. The original formulation of the theorem allows us to understand whether $f \in \mathcal{F}(\mathbb{L})$ according to how its rate of decay relates to the decays of functions of the type $\xi^{-\alpha}$. Interpretations in terms of $\tau$ and $\phi$ both make the same two improvements: they provide an extra hypothesis that closes the aforementioned gap (and as such allows as to show that either $f$ belongs to $\mathcal{F}(\mathbb{L})$ or it does not), and they provide intelligent ways of generating functions which clearly do or do not satisfy the hypotheses of our theorems.

Let us begin with $\phi$. We have the following theorem that closes the gap. The extra hypothesis (the concavity of $\phi$) is analogous to the concavity hypothesis on $h$ in Theorem 1 of \cite{ols05}, which is a weaker version of Theorem \ref{thm:Olsen-Renfro}.

\begin{theorem}\label{thm:phi-desambiguates}
     Let $\phi$ be defined as in Theorem \ref{thm:equivalent-decay-conditions}. If $\phi$ is concave, then one and only one of the scenarios \eqref{eq:decay-condition} and \eqref{eq:decay-condition-2} occurs, and as such the membership of $f$ in $\mathcal{F}(\mathbb{L})$ is completely determined by the respective limits.
\end{theorem}

\begin{proof}
    It is clear that \eqref{eq:decay-condition} and \eqref{eq:decay-condition-2} are mutually exclusive, so we will be done by assuming that \eqref{eq:decay-condition} is false and proving that \eqref{eq:decay-condition-2} must then be true.

It is useful to state the following lemma, whose proof we omit because it is straightforward.

\begin{lemma}\label{lem:LS-LI}
Let $f: \mathbb{R}_{\ge0} \longrightarrow {\mathbb{R}_{\ge0}}$ be a bounded function and let
\begin{equation*}
LS(\alpha) := \limsup_{\xi \rightarrow +\infty} \frac{\xi^{-\alpha}}{f(\xi)}
    \qquad , \qquad  \ LI(\alpha) := \liminf_{\xi \rightarrow +\infty} \frac{\xi^{-\alpha}}{f(\xi)}.
\end{equation*}

Then the following are true:
\begin{enumerate}[(i)]
    \item $LI(\alpha) \le LS(\alpha)$.
    \item $LI(\alpha_0)>0 \implies LI(\alpha) = +\infty$ for all $0 \le \alpha < \alpha_0$.
    \item $LI(\alpha_0)< +\infty \implies LI(\alpha)=0$ for all $\alpha > \alpha_0$.
    \item $LS(\alpha_0)>0 \implies LS(\alpha) = +\infty$ for all $0 \le \alpha < \alpha_0$.
    \item $LS(\alpha_0)< +\infty \implies LS(\alpha)=0$ for all $\alpha > \alpha_0$.
\end{enumerate}
\end{lemma}

\begin{figure}[ht]
\centering
    \includegraphics[width=\textwidth]{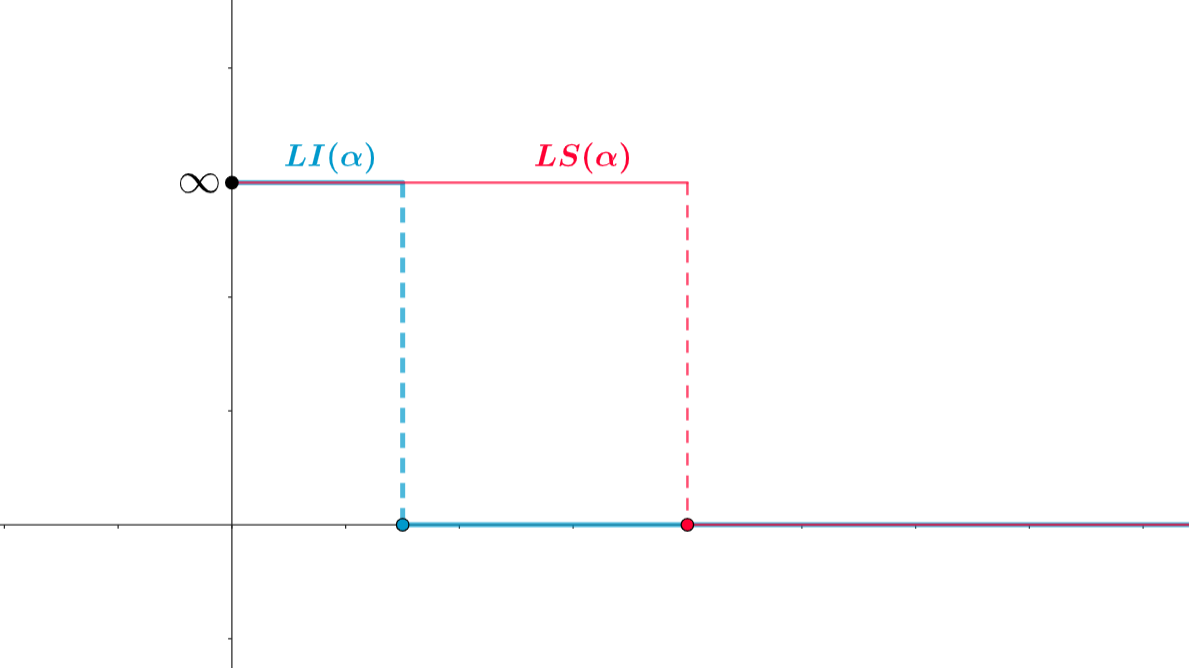}
    \caption{The graphs of the functions $LS$ and $LI$ for some function $f$.}
    \label{fig:limit-inferior-superior}
\end{figure}

In virtue of this lemma, the assumption that \eqref{eq:decay-condition} is false can be expressed as the existence of an $\alpha_0 > 0$ such that
\begin{equation*}
    \limsup_{\xi \rightarrow +\infty} \frac{\xi^{-\alpha_0}}{f(\xi)} = +\infty,
\end{equation*}
and we will be done if we prove that
\begin{equation*}
    \liminf_{\xi \rightarrow +\infty} \frac{\xi^{-\alpha_0}}{f(\xi)} = +\infty,
\end{equation*}

We will restate both these equalities by applying the change of variables $\xi = e^{\gamma}$ that we already used. Our hypothesis will now be that
\begin{equation*}
    + \infty =  \limsup_{\xi \rightarrow +\infty} \frac{\xi^{-\alpha_0}}{f(\xi)} = \limsup_{\gamma \rightarrow +\infty} \frac{e^{-\alpha_0 \gamma}}{e^{-\phi(\gamma)}} = \limsup_{\gamma \rightarrow +\infty} e^{-\alpha_0 \gamma + \phi(\gamma)},
\end{equation*}
or equivalently, that
\begin{equation}\label{eq:1-false-alternative}
    \limsup_{\gamma \rightarrow +\infty} -\alpha_0 \gamma + \phi(\gamma) = +\infty.
\end{equation}
Similarly, what we want to prove is that
\begin{equation}\label{eq:2-false-alternative}
    \liminf_{\gamma \rightarrow +\infty} -\alpha_0 \gamma + \phi(\gamma) = + \infty.
\end{equation}
In virtue of \eqref{eq:1-false-alternative}, pick $(\gamma_k)_k$ a sequence of positive numbers that sctrictly increases to $+\infty$ such that $-\alpha_0 \gamma_k +\phi(\gamma_k) \ge k$. Given any $\gamma \ge \gamma_1$, we can find $k \in \mathbb{N}, t \in [0,1]$ such that $\gamma = t\gamma_k + (1-t)\gamma_{k-1}$. We therefore have, as a consequence of the concavity of $\phi$, that
\begin{align*}
    -\alpha_0 \gamma +\phi(\gamma) &= -t\alpha_0 \gamma_k -(1-t)\alpha_0\gamma_{k+1}
    +\phi(t\gamma_k +(1-t)\gamma_{k+1}) \\
    &\ge t(-\alpha_0\gamma_k+\phi(\gamma_k)) +(1-t)(-\alpha_0 \gamma_{k+1} + \phi(\gamma_{k+1})) \ge k.
\end{align*}
Since this is true for every $\gamma \ge \gamma_1$, it follows that \eqref{eq:2-false-alternative} holds.
\end{proof}

An analogous result can be proved when $\phi$ is convex, but actually more can be said in this case.

\begin{theorem}\label{thm:when-phi-convex}
    Let $\phi$ be defined as in Theorem \ref{thm:equivalent-decay-conditions}. If $\phi$ is convex and not constant, then $f \notin \mathcal{F}(\mathbb{L})$. 
\end{theorem}

\begin{proof}
    Again, assuming that $\| f \|_\infty \le \frac{1}{2} $ we have that $\phi(\gamma) \ge 0$. Notice that, since $\phi$ is convex and bounded below, it is increasing.
    
    Under all these assumptions define, for some $\gamma_0 > 0$ to be fixed later, $\tilde{\phi}(\gamma) = \phi(\gamma+\gamma_0)-\phi(\gamma_0)$. It is clear that $\tilde{\phi}$ is also convex and $\tilde{\phi}(0)=0$. As such, $\tilde{\phi}(\lambda \gamma) \ge \lambda \tilde{\phi}(\gamma)$ for every $\lambda \ge 1$ and $\gamma \ge 0$.

    We therefore have that
    \begin{align*}
        \liminf_{\gamma \rightarrow +\infty} \frac{\phi(\gamma)}{\gamma} &= \liminf_{\gamma \rightarrow +\infty } \frac{\phi(\gamma+\gamma_0)}{\gamma+\gamma_0} \\
        &= \liminf_{\gamma \rightarrow +\infty} \frac{\tilde{\phi}(\gamma)+\phi(\gamma_0)}{\gamma+\gamma_0} \\
        &\ge \liminf_{\gamma \rightarrow +\infty} \frac{\gamma \tilde{\phi}(1)+\phi(\gamma_0)}{\gamma+\gamma_0} = \tilde{\phi}(1) = \phi(1+\gamma_0)-\phi(\gamma_0),
    \end{align*}
    where in the inequality we assume $\gamma>1$. Since $\phi$ is increasing and not constant, we can choose $\gamma_0$ such that $\phi(1+\gamma_0)>\phi(\gamma_0)$. As such, \eqref{eq:phi-guarantees-outside} holds and $f \notin \mathcal{F}(\mathbb{L})$.
\end{proof}

The other benefit of $\phi$ is that it allows us to create functions that belong to $\mathcal{F}(\mathbb{L})$: simply pick a function $\phi(\gamma) = o(\gamma)$ and invert the construction, letting $f(\xi) = e^{- \phi(\log (\xi))}$. As it turns out, the slower that $\frac{\phi(\gamma)}{\gamma}$ decays to 0, the faster $f$ does, and as such we are able to construct functions in $\mathcal{F}(\mathbb{L})$, with arbitrarily fast decays inside of what is allowed. In a similar fashion $\phi$ can be used to construct functions which do not belong to $\mathcal{F}(\mathbb{L})$, or that belong to the gap.

\begin{center}
    \begin{table}[h!]
\label{table:phi-and-f}
    \begin{tabular}{|P{3cm}|P{7cm}|}
    \hline\xrowht[()]{20pt}
 $\phi(\gamma)$ & $f(\xi) $ \\ 
 \hline\xrowht[()]{20pt}
 $\alpha \gamma$ & $\xi^{-\alpha}\ \  (\notin \mathcal{F}(\mathbb{L}))$ \\ 
 \hline\xrowht[()]{20pt}
 $\log \gamma$ & $\displaystyle \frac{1}{\log \xi} \ \ (\in \mathcal{F}(\mathbb{L}))$ \\ 
 \hline\xrowht[()]{20pt}
 $\gamma^p$ & $\displaystyle e^{-(\log \xi)^p}\ \  (\in \mathcal{F}(\mathbb{L})$ iff $p<1)$\\ 
 \hline\xrowht[()]{20pt}
 $\displaystyle \frac{\gamma}{\log \gamma}$ & $ \displaystyle e^{-\frac{\log \xi}{\log \log \xi}} = \xi^{-\frac{1}{\log \log \xi}} \ \ (\in \mathcal{F}(\mathbb{L}))$\\ 
 \hline
    \end{tabular}
    
    \vspace{.5cm}
    
    \caption{Examples of functions $\phi$ and their correspondent decays $f$. }
    \end{table}
\end{center}

We now turn our attention to $\tau$. We have previously proved that the concavity or convexity of $\phi$ closes the gap, and we cited that Olsen used a similar hypothesis in a weaker version of Theorem \ref{thm:Olsen-Renfro}. We can prove the following result that resembles the effect of Olsen's hypothesis in Theorem 2 of \cite{ols05}. 
\begin{theorem}\label{thm:tau-disambiguates}
    Let $f$ be as in Theorem \ref{thm:fourier-set-liouville-decreasing}. If for some $\xi_0 >0$
    \begin{equation}\label{eq:tau-disambiguates-hypothesis}
        f(\xi^{\lambda}) \ge f(\xi)^{\lambda} \quad \forall \xi \ge \xi_0, \lambda \ge 1
    \end{equation}
    holds, then one and only one of the scenarios \eqref{eq:tau-guarantees-inside} and \eqref{eq:tau-guarantees-outside} occurs, and as such the membership of $f$ in $\mathcal{F}(\mathbb{L})$ is completely determined by the respective limits. 
\end{theorem}

\begin{proof}
Assume without loss of generality that $\| f\|_{\infty}< \frac{1}{2}$. 

Let $\xi_2 \ge \xi_1 > \max \{\xi_0,1 \}$. Apply \eqref{eq:tau-disambiguates-hypothesis} for $\xi = \xi_1$, $\lambda = \frac{\log \xi_2}{\log \xi_1}$. We get
    \begin{align*}
        f(\xi_1^{\frac{\log \xi_2}{\log \xi_1}}) &\ge f(\xi_1)^{\frac{\log \xi_2}{\log \xi_1}} \implies \\
        f(\xi_2) &\ge f(\xi_1)^{\frac{\log \xi_2}{\log \xi_1}} \implies \\
        \log (f (\xi_2)) &\ge \frac{\log \xi_2}{\log \xi_1} \log (f(\xi_1)) \implies \\
        -\frac{\log(f(\xi_2))}{\log(\xi_2)} &\le -\frac{\log(f(\xi_1))}{\log(\xi_1)} \implies \\
        \tau(\xi_2) &\le \tau(\xi_1).
    \end{align*}
Since $\tau$ is eventually decreasing, it has a limit at infinity. Since $\tau$ is eventually positive, the limit is either 0 or greater than 0. In the former case \eqref{eq:tau-guarantees-inside} holds, in the latter \eqref{eq:tau-guarantees-outside} holds.
\end{proof}

Note that a similar calculation can be made for functions for which \eqref{eq:tau-disambiguates-hypothesis} is satisfied with the opposite inequality. However, in this case we would get that $\tau$ is positive and eventually increasing, and as such its limit must be greater than 0. The result therefore is
\begin{theorem}\label{thm:tau-forces-out}
     Let $f$ be as in Theorem \ref{thm:fourier-set-liouville-decreasing}. If for some $\xi_0 >0$
    \begin{equation*}\label{eq:tau-forces-out-hypothesis}
        f(\xi^{\lambda}) \le f(\xi)^{\lambda} \quad \forall \xi \ge \xi_0, \lambda \ge 1
    \end{equation*}
    holds, then $f \notin \mathcal{F}(\mathbb{L})$.
\end{theorem}

The other utility of $\tau$ is that, as with $\phi$, it allows us to reinterpret the functions $f$ in terms which make it easy to construct them with a desired decay at infinity. Recall the definition of $\tau$ and rearrange it to obtain $f(\xi) = \xi^{- \tau(\xi)}$. Note that the case of $\tau(\xi)$ being constant is of no interest, since we already know that power-like decays are not in $\mathcal{F}(\mathbb{L})$. Hence, allowing non constant exponents $\tau(\xi)$ will contribute to cover more interesting cases.

So if we pick any function $\tau$ which satisfies either \eqref{eq:tau-guarantees-inside} or \eqref{eq:tau-guarantees-outside}, we get a function $f$ whose membership in $\mathcal{F}(\mathbb{L})$ we understand. We can also choose a function $\tau$ such that
\begin{equation*}
    \liminf_{\xi \rightarrow +\infty} \tau(\xi) = 0 \quad \textnormal{and} \quad \limsup_{\xi \rightarrow + \infty } \tau(\xi) >0
\end{equation*}
and get a function whose behaviour is not explained by any of the results we have proved so far. For example, by letting $\tau(\xi) = |\cos(\pi \xi)|$ we get
\begin{equation*}
    f(\xi) = \xi^{- |\cos(\pi \xi)|},
\end{equation*}
a function in the gap generated by Theorem \ref{thm:fourier-set-liouville-decreasing}. We will be able to prove in Propositon  \ref{prop:simple-cases-in-out} that $f \notin \mathcal{F}(\mathbb{L})$, but that will require the study of a different aspect of the Liouville set, which we will tackle in the following section.

\section{Shift invariance and periodic multipliers}\label{sec:SI-PM}

In this section we present the proof of Theorem \ref{thm:multiplier-fourier-series-period1} which will turn out to be of great importance for identifying if a function in the gap is or is not in $\mathcal{F}(\mathbb{L})$. As mentioned in the introduction, this has to do with a key  property of $\mathbb{L}$ being invariant under integer translations.

Indeed, let $\mu$ be a measure supported on $\mathbb{L}$, and let $k \in \mathbb{Z}$. Then the measure $\mu_k$ defined by
    \begin{equation*}
        \mu_k(A) = \mu(A-k)
    \end{equation*}
    is also a measure supported on $\mathbb{L}$. In particular, we have that
    \begin{equation*}
        \widehat{\mu_k}(\xi) = e^{-2\pi i \xi k} \hat{\mu} (\xi)
    \end{equation*}
and $\mu_k$ is supported on $\mathbb{L}$. This allows for some interesting constructions. An immediate consequence of the definition above is that the measure $\nu$ defined as
\begin{equation*}
        \nu(E) = \frac{1}{2} \left(\mu_k(E) + \mu_{-k}(E)\right)
    \end{equation*} 
is supported on $\mathbb{L}$ and its Fourier transform is given by
\begin{equation*}
        \widehat{\nu}(\xi) = \cos(2\pi k \xi) \hat{\mu}(\xi).
    \end{equation*}

This simple yet crucial observation can be used to obtain Theorem \ref{thm:multiplier-fourier-series-period1}. We present here the proof.

\begin{proof}[Proof of Theorem \ref{thm:multiplier-fourier-series-period1}]
We begin by noticing that the hypotheses on $g$ imply that $g(0) > 0$. Identify the decay $f$ with some representative and let $\mu$ be a probability measure supported on $\mathbb{L}$ such that $|\hat{\mu}(\xi)| = O(f(|\xi|))$. Let $a_k\ge 0$ be the sequence of Fourier coefficients of $g$. Consider the partial sum
    \begin{equation*}
        S_N(\xi) = \sum_{k=0}^N a_k \cos(2k\pi \xi)
    \end{equation*}
    and define the measures
    \begin{equation*}
        \mu^{(N)}(E) = \sum_{k=0}^{N}  \frac{a_k}{2}(\mu_{k}(E)+\mu_{-k}(E)).
    \end{equation*}
    Then, $\mu^{(N)}$ is a measure supported on $\mathbb{L}$ as a consequence of the fact that $\mu$ is a measure supported on $\mathbb{L}$ and the fact that $a_k \ge 0$ for every $k$. Also by construction we get that \begin{equation*}
        \widehat{\mu^{(N)}}(\xi) = \left( \sum_{k=0}^{N} a_k \cos(2 \pi k \xi)\right) \widehat{\mu}(\xi).
    \end{equation*}
    Note that 
    \begin{equation*}
        0 < g(0) = \sum_{k=0}^{\infty} a_k  
    \end{equation*}
    and since $a_k \ge 0$, there is some $N_0$ such that $N \ge N_0$ implies
    \begin{equation*}
        \mu^{(N)}(\mathbb{L}) = \widehat{\mu^{(N)}}(0) = \sum_{k=0}^{N} a_k >0.
    \end{equation*}
As a consequence, $(\mu^{(N)})_N$ is a sequence of measures which are all finite and nonzero from some point onwards. 
    
Also $\widehat{\mu^{(N)}}$ are all (uniformly) continuous functions (they are Fourier transforms of finite measures) and for every $N>M$ we have that
\begin{align*}
\left|\widehat{\mu^{(N)}}(\xi)-\widehat{\mu^{(M)}}(\xi) \right| & = \left|\sum_{k=M+1}^{N} a_k \cos(2 \pi k \xi)\right| \left|\widehat{\mu}(\xi)\right|\\
        & \le \sum_{k=M+1}^{\infty} a_k \xrightarrow{M \rightarrow +\infty} 0.
\end{align*}
This implies that $(\widehat{\mu^{(N)}})_N$ is a Cauchy sequence with respect to $\| \|_{\infty}$. As a consequence, its limit is continuous. But hypothesis $(4)$ guarantees that $$\widehat{\mu^{(N)}}(\xi) \xrightarrow{N \rightarrow \infty} g(\xi) \hat{\mu}(\xi)$$  for every $\xi$. The Lévy Continuity Theorem therefore guarantees the existence of a finite measure $\mu^{(\infty)}$, supported on $\mathbb{L}$, which is the weak limit of $(\widehat{\mu^{(N)}})_N$ and for which 
    \begin{equation*}
        \widehat{\mu^{(\infty)}}(\xi) = g(\xi) \hat{\mu}(\xi).
    \end{equation*}
    Note that $\mu^{(\infty)}$ is not the zero measure since $\widehat{\mu^{(\infty)}}(0) = g(0) \hat{\mu}(0) >0$.
The result follows from the fact that $|\hat{\mu}(\xi)|=O(f(|\xi|))$.

\end{proof}

Now that the proof is done, we would like to point out some things about the hypotheses of Theorem \ref{thm:multiplier-fourier-series-period1}. First, recall that a sufficient condition so that hypothesis $(4)$ holds is that $g \in L^1([-\frac{1}{2},\frac{1}{2}])$ and is everywhere differentiable. This gives us a big family of functions to draw from. Second, condition $(3)$ is clearly necessary to carry out the proof since we need to construct a sequence of measures to obtain the desired limit measure.

Now that we know that functions $f$ in $\mathcal{F} (\mathbb{L})$ can be multiplied by certain periodic functions $|g|$ to obtain functions $|g|f$ still in $\mathcal{F} (\mathbb{L})$, we would like to generate useful examples of functions $g$. The following result will guarantee the existence of bump functions that satisfy the hypotheses of Theorem \ref{thm:multiplier-fourier-series-period1}.

\begin{proposition}\label{prop:existence-good-bumps-period1}
Let $\gamma:[-\frac{1}{2}, \frac{1}{2}] \longrightarrow \mathbb{R}$ be a bump function satisfying the following hypotheses:
\begin{enumerate}[(i)]
    \item $\gamma$ is even.
    \item $\gamma \in C^1[-\frac{1}{2},\frac{1}{2}]$.
    \item $\gamma(0)=1, 0 \le \gamma(t) \le 1 \ \forall \ t \in [-\frac{1}{2}, \frac{1}{2}]$.
    \item supp($\gamma$) $\subset [-\delta,\delta] $ for some $\delta < \frac{1}{4}$.
    
\end{enumerate}
Then the function
\begin{equation*}
    g = \frac{\gamma \ast \gamma}{\| \gamma\|_2^2}
\end{equation*}
is also a bump function satisfying hypotheses $(i)-(iii)$ but also
\begin{enumerate}
    \item[(iv)] supp($g$) $\subset [-2\delta,2\delta] $.
    \item[(v)] When $g$ is extended to $\mathbb{R}$ as a $1$-periodic function, all its Fourier coefficients are nonnegative. 
\end{enumerate}
    
\end{proposition}

\begin{proof}
Let us start by mentioning that the existence of a function $\gamma$ satisfying the hypotheses $(i)$ to $(iv)$ is an elementary and well known fact.
Let us now check all five conditions for $g$. To check $(i)$ notice that
    \begin{align*}
        g(-x) &= \frac{1}{\left\| \gamma \right\|_2^2} \int_{-1/2}^{1/2} \gamma(t) \gamma(-x-t) dt \\
        &=  \frac{1}{\left\| \gamma \right\|_2^2} \int_{-1/2}^{1/2} \gamma(-t) \gamma(-x+t) dt \\
        &=  \frac{1}{\left\| \gamma \right\|_2^2} \int_{-1/2}^{1/2} \gamma(t) \gamma(x-t) dt = g(x),
    \end{align*}
where in the second equality we have applied the change of variables $t \mapsto -t$, and in the third equality we have used the evenness of $\gamma$.

Condition $(ii)$ is a known result on the preservation of continuity and differentiability through convolutions. $(iii)$ follows directly from the evenness of $\gamma$:
\begin{equation*}
    g(0) = \frac{1}{\| \gamma \|_2^2} \int_{-1/2}^{1/2} \gamma(t) \gamma(-t) dt = 1.
\end{equation*}

$(iv)$ can be obtained as follows:
\begin{align*}
    g(x)>0 &\implies \int_{-1/2}^{1/2} \gamma(t) \gamma(x-t) dt >0 \\
    &\implies \exists \ t_0 \in [-\delta, \delta]/ \ \gamma(t_0)>0 \textnormal{\ and \ } \gamma(x-t_0) > 0 \\
    &\implies x-t_0 \in [-\delta,\delta] \implies x \in [-2\delta,2\delta].
\end{align*}

Finally, the proof of $(v)$ begins with the following calculation:
\begin{equation}\label{eq:first-calculation-proof-vi}
    \| \gamma \|_2^2 . \widehat{g}(n) = 
   \| \gamma \|_2^2 \int_{-1/2}^{1/2} g(x) \cos(nx)dx = 
    \int_{-1/2}^{1/2} \int_{-1/2}^{1/2} \gamma(x-t) \gamma(t) \cos(nx) dt dx.
\end{equation}
Now we note that
\begin{equation*}
    \cos(nx) = \cos(n(x-t)+nt) = \cos(n(x-t))\cos(nt) - \sin(n(x-t))\sin(nt).
\end{equation*}
We can use this, alongside Fubini's theorem, to rewrite \eqref{eq:first-calculation-proof-vi} as
\begin{align*}
     \| \gamma \|_2^2 . \widehat{g}(n) &= \int_{-1/2}^{1/2} \int_{-1/2}^{1/2} \gamma(x-t) \gamma(t) \cos(n(x-t))\cos(nt) dx dt \\
     &= \int_{-1/2}^{1/2} \gamma(t) \cos(nt) \left[ \int_{-1/2}^{1/2} \gamma(x-t)\cos(n(x-t)) dx \right] dt = \widehat{\gamma}(n)^2 \ge 0.
\end{align*}
We can omit the term involving the sines from the calculation, since it will devolve into the integral of an odd function over $[-\frac{1}{2},\frac{1}{2}]$. 

\end{proof}

We can use these tools to tackle the problem that remained from the previous section: deciding whether a function in the gap left by Theorem \ref{thm:fourier-set-liouville-decreasing} belongs to $\mathcal{F}(\mathbb{L})$. Everything that we have developed in this section will serve to analyze the two examples described in the introduction in equations \eqref{eq:example-cos/log} and \eqref{eq:example-xi^cos}.

\begin{proposition}\label{prop:simple-cases-in-out}
\mbox{}
\begin{enumerate}
    \item[i)] Let $h(\xi) = \frac{|\cos(2 \pi \xi)|}{\log(\xi)}$. Then 
$h \in \mathcal{F}(\mathbb{L})$. 
\item[ii)] Let $f(\xi) = \xi^{-|\cos(\pi \xi)|}$. Then 
$f \notin \mathcal{F}(\mathbb{L})$. 
\end{enumerate}
\end{proposition}

\begin{figure}[ht]
\centering
    \includegraphics[width=\textwidth]{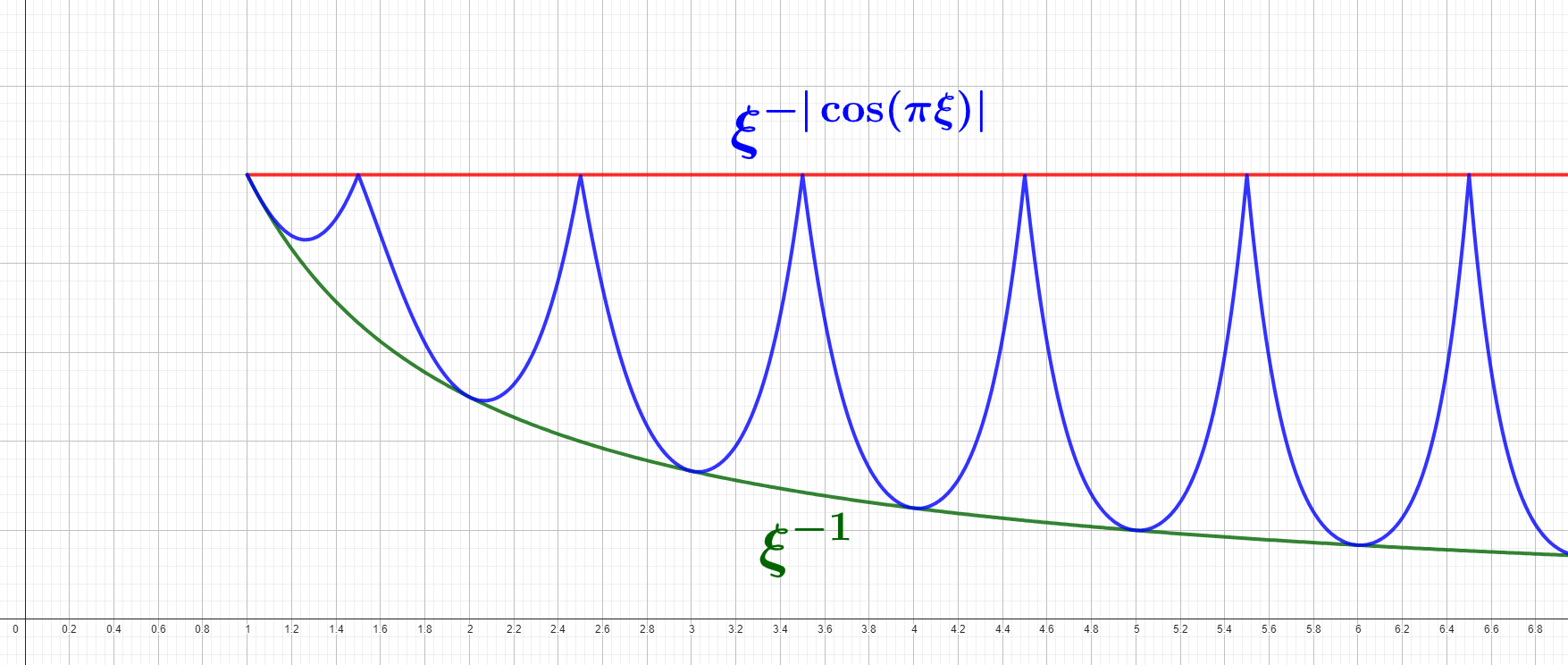}
    \caption{The graph of $f(\xi) = \xi^{-|\cos(\pi \xi)|}$ and the two functions between which it fluctuates.}
    \label{fig:gap-example1}
\end{figure}

\begin{proof}
Let us start by proving $i)$. We start by choosing a measure $\mu$ supported on $\mathbb{L}$ such that 
$$
|\widehat{\mu}(\xi)|=O\left(\log(\xi)^{-1}\right).
$$
Now, define the measure $\nu$ as
\begin{equation*}
        \nu(E) = \frac{1}{2} \left(\mu_1(E) + \mu_{-1}(E)\right).
    \end{equation*} 
By the translation invariance of $\mathbb{L}$, we know that $\nu$ is a measure supported on it. Also, its Fourier transform is given by
\begin{equation*}
        \widehat{\nu}(\xi) = \cos(2\pi\xi) \hat{\mu}(\xi)=O\left(\cos(2\pi\xi)\log(\xi)^{-1}\right),
\end{equation*}
which yields the desired conclusion.

To prove $ii)$, suppose that indeed $f \in \mathcal{F}(\mathbb{L})$. 
The graph of $f$ intersects that of $\xi^{-1/2}$ at the points whose $\xi$ coordinate satisfies the equation
 
\begin{equation*}
        \xi^{-|\cos(\pi \xi)|} = \xi^{-1/2}.
    \end{equation*}
    A straightforward calculation shows that these coordinates are precisely those of the form $\xi = k \pm \frac{1}{3}, k \in \mathbb{N}$.

       \begin{figure}[ht]
\centering
    \includegraphics[width=\textwidth]{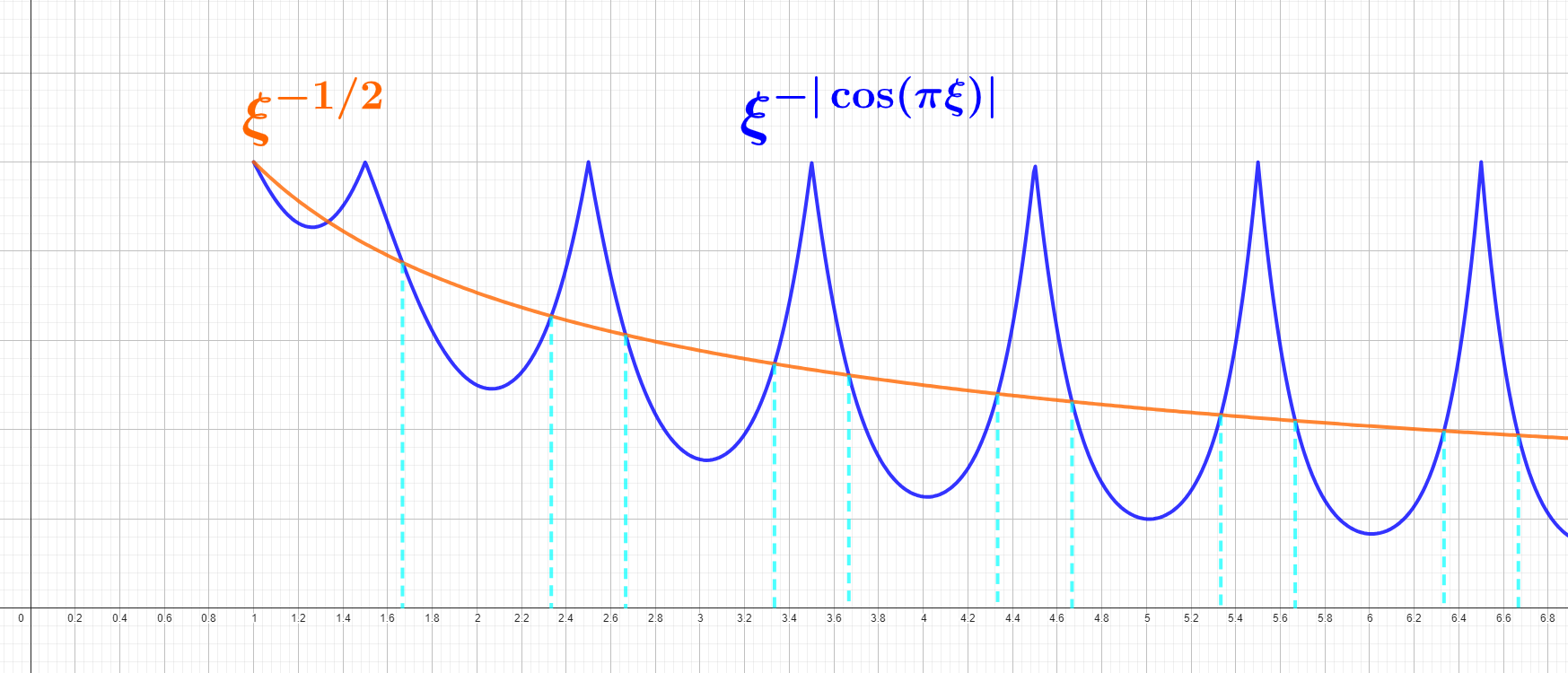}
    \caption{The graph of $f$. The intervals of radius $\frac{1}{3}$ surrounding the integers determine subsets of the domain where $f$ is bounded by $\xi^{-1/2}$.}
    \label{fig:gap-example2}
\end{figure}

Let us start by picking a 1-periodic bump function $g \in C^1(\mathbb{R})$ as in Theorem \ref{thm:multiplier-fourier-series-period1} satisfying $\textnormal{supp}(g) \cap [-\frac{1}{2}, \frac{1}{2}] \subset [-\frac{1}{3},\frac{1}{3}] $. It is clear that $gf(\xi) \le \xi^{-1/2}$ for $\xi \notin \textnormal{supp}(g)$. If on the other hand $\xi \in \textnormal{supp}(g)$, because we have constructed $g$ to be supported where $f$ is bounded by $\xi^{-1/2}$, and because $g\le 1$ everywhere, we have that $gf(\xi) \le \xi^{-1/2}$ for every $\xi \in \textnormal{supp}(g)$. All in all, we get that $gf(\xi) \le \xi^{-1/2} $. Our assumption on $f$ together with Theorem \ref{thm:multiplier-fourier-series-period1} implies that $gf \in \mathcal{F}(\mathbb{L})$ but this in turn implies that $\xi^{-1/2} \in \mathcal{F}(\mathbb{L})$, which is clearly false and thus shows the contradiction. 

\end{proof}

While the proof of $(ii)$ is extremely tailored, it does shed some light on the reasons why a function in the gap might fail to belong to $\mathcal{F}(\mathbb{L})$. There are two key ingredients in this proof: the first is that points of fast decay (in this case, where $f$ behaves like $\xi^{-1}$) appear with linear periodicity (that is, they are located on $\mathbb{N}$, each separated from the previous one by 1). The second ingredient is a bit more subtle. Note that by looking at intervals of radius $\frac{1}{3}$ around \textbf{each} natural number we get the desired bound, but it was mandatory that the interval had the same radius for each natural number. The proof worked because the behaviour around each ``valley'' is similar regardless of which valley we are at. It would not work if the valleys were increasingly narrow, because we would need to take smaller intervals as we advance on $\mathbb{N}$, which is incompatible with the periodicity of $g$.

We will now attempt to generalize Proposition \ref{prop:simple-cases-in-out}. This will still be far from closing the gap, but we will at least get close to formalizing the ideas from the previous paragraph.

Let $a>0$ be some real number. Define the set $\mathbb{L}_a = \mathbb{L} + \frac{1}{a}\mathbb{Z}$. Since $\mathbb{L}_a$ is the union of a countable number of copies of $\mathbb{L}$, $\dim(\mathbb{L}_a)=0$. Also, it is clear from the definition that $\mathbb{L}_a$ is invariant under translations by numbers of the form $\frac{m}{a}, m \in \mathbb{Z}$.

We can therefore proceed in similar fashion to what we have done throughout this section and prove the following results. We state them without proof since they are entirely analogous to the ones that prove the results for $\mathbb{L}$.

\begin{proposition}\label{prop:multiplier-fourier-series-period-a}
    Let $f \in \mathcal{F}(\mathbb{L}_a)$. Let $g: \mathbb{R} \rightarrow \mathbb{R}$ be a function which satisfies the following conditions:
    \begin{enumerate}[(1)]
        \item $g$ is even and $a$-periodic.
        \item $g$ is not identically equal to 0.
        \item $g$ has an $a$-periodic Fourier series $S_N(\xi)$ with nonnegative coefficients.
        \item $S_N$ converges to $g$ pointwise.
    \end{enumerate}
    Then $|g|f \in \mathcal{F}(\mathbb{L}_a)$.
\end{proposition}

\begin{proposition}\label{prop:existence-good-bumps-period-a}
There exists a bump function $g: [-\frac{a}{2},\frac{a}{2}] \longrightarrow \mathbb{R}$ for which the following properties hold:
\begin{enumerate}[(i)]
    \item $g$ is even.
    \item $g \in C^1[-\frac{a}{2},\frac{a}{2}]$.
    \item $g(0)=1, 0 \le g(t) \le 1 \ \forall \ t \in [-\frac{a}{2}, \frac{a}{2}]$.
    \item supp($g$) $\subset [-\delta,\delta] $ for some $\delta < \frac{a}{2}$.
    \item When $g$ is extended to $\mathbb{R}$ as an $a$-periodic function, all its Fourier coefficients are nonnegative.
\end{enumerate}
\end{proposition}

This allows us to prove the following result.

\begin{proposition}\label{prop:general-condition-outside}
    Let $f: \mathbb{R}^+_{0} \longrightarrow \mathbb{R}^+$ and let $\tau$ be as in Theorem \ref{thm:equivalent-decay-conditions} such that
    \begin{enumerate}[(i)]
        \item $\displaystyle \liminf_{k\rightarrow +\infty}  \tau(ak) > 0$ ($f$ has fast decay on $a\mathbb{Z}$).
        \item $\tau$ is uniformly continuous.
    \end{enumerate}
    Then $f \notin \mathcal{F}(\mathbb{L})$.
\end{proposition}

\begin{proof}
    Suppose that $f \in \mathcal{F}(\mathbb{L}) \subset \mathcal{F}(\mathbb{L}_a)$. Let $\liminf_{k \rightarrow + \infty} \tau(ak) = \alpha_0 >0$. In virtue of the uniform continuity of $\tau$, there exists a positive number $\delta < \frac{a}{2}$ such that
    \begin{equation*}
        |\xi_1 - \xi_2 | \le \delta  \implies |\tau(\xi_1) - \tau(\xi_2)| < \frac{\alpha_0}{4}.
    \end{equation*}
In addition, let $k_0 \in \mathbb{N}$ such that $k \ge k_0 \implies \tau(ak) > \frac{\alpha_0}{2}$. In consequence, for $\xi \in [ak-\delta, ak+\delta], k \ge k_0$ we get that
\begin{equation*}
    \tau(\xi) = \tau(\xi)-\tau(ak)+\tau(ak) > -\frac{\alpha_0}{4} +\frac{\alpha_0}{2} = \frac{\alpha_0}{4}.
\end{equation*}

Let $g$ be an $a$-periodic bump function as in Proposition \ref{prop:existence-good-bumps-period-a} which satisfies that $\textnormal{supp}(g) \cap [-\frac{a}{2},\frac{a}{2}] = [-\delta,\delta]$. We therefore get that $gf \in \mathcal{F}(\mathbb{L}_a)$, but
\begin{equation*}
    gf(\xi) = \left\{ \begin{array}{lcc} 0 \ \textnormal{if} \ \xi \notin \bigcup_{k \in \mathbb{N}_0}(ak-\delta,ak+\delta)   \\ \\ 
    \xi^{-\tau(\xi)} \le \xi^{-\alpha_0 /4} \ \textnormal{if} \ \xi \in [ak-\delta, ak+\delta], k \ge k_0 .  \end{array} \right.
\end{equation*}

But this implies $\xi^{-\alpha_0/4} \in \mathcal{F}(\mathbb{L}_a) \implies \dim(\mathbb{L}_a) > 0$, which is false.
\end{proof}

\section{Acknowledgements}
This research was supported by grants: PICT 2018-3399
(ANPCyT), PICT 2019-03968 (ANPCyT) and CONICET PIP 11220210100087.

\

\bibliographystyle{amsalpha}

\providecommand{\bysame}{\leavevmode\hbox to3em{\hrulefill}\thinspace}
\providecommand{\MR}{\relax\ifhmode\unskip\space\fi MR }
\providecommand{\MRhref}[2]{%
  \href{http://www.ams.org/mathscinet-getitem?mr=#1}{#2}
}
\providecommand{\href}[2]{#2}

\end{document}